\documentclass[reqno, 12pt,a4letter]{amsart}
\usepackage{amsmath,amsxtra,amssymb,latexsym, amscd,amsthm}
\usepackage[mathscr]{euscript}
\usepackage{mathrsfs}
\usepackage[english]{babel}
\usepackage[active]{srcltx}
\usepackage{enumerate}
\usepackage{color}
\usepackage{mathtools,stmaryrd}
\usepackage[T1]{fontenc}

\newtheorem {theorem}{Theorem}[section]
\newtheorem {lemma}{Lemma}[section]
\newtheorem {corollary}{Corollary}[section]

\theoremstyle{definition}
\newtheorem{definition}{Definition}[section]
\newtheorem{remark}{Remark}[section]

\newcommand{\rank}{{\rm rank}}
\newcommand{\aff}{{\rm aff}}
\newcommand{\ri}{{\rm ri}}

\def\EES{{\accent"5E e}\kern-.5em\raise.8ex\hbox{\char'23 }}
\def\ow{o\kern-.42em\raise.82ex\hbox{
   \vrule width .12em height .0ex depth .075ex \kern-0.16em \char'56}\kern-.07em}
\def\OW{o\kern-.460em\raise1.36ex\hbox{
\vrule width .13em height .0ex depth .075ex \kern-0.16em
\char'56}\kern-.07em}
\def\DD{D\kern-.7em\raise0.4ex\hbox{\char '55}\kern.33em}

\title{Stability of closedness of closed convex sets under linear mappings}

\author{S\~i Ti\d{\^e}p \DD inh$^\dag$}
\address{Institute of Mathematics, VAST, 18, Hoang Quoc Viet Road, Cau Giay District 10307, Hanoi, Vietnam}
\email{dstiep@math.ac.vn}

\author[T. S. Ph\d{a}m]{Ti\'{\^{e}}n-S\OW n Ph\d{a}m$^\ddag$}
\address{Department of Mathematics, Dalat University, 1 Phu Dong Thien Vuong, Dalat, Vietnam}
\email{sonpt@dlu.edu.vn}

\thanks{$^\dag$The first author is partially supported by Vietnam National Foundation for Science and Technology Development (NAFOSTED) grant 101.04-2019.305}
\thanks{$^{\ddag}$The second author is partially supported by Vietnam National Foundation for Science and Technology Development (NAFOSTED), grant 101.04-2019.302.}



\date{ \today}

\begin{document}

\begin{abstract}
We study the problem of when the linear image of a fixed closed convex set $X \subset\mathbb{R}^n$ is closed. Specifically, we improve the main results in the papers \cite{Borwein2009, Borwein2010} by showing that for almost all linear mappings $T$ from $\mathbb{R}^n$ into $\mathbb{R}^m,$ not only $T(X)$ is closed, but there is also an open neighborhood of $T$ whose members also preserve the closedness of $X.$
\end{abstract}

\maketitle

\pagestyle{plain}

\section{Introduction}

We consider the question of when the linear image of a closed convex set in $\mathbb{R}^n$ is closed. The closedness of such images is of significance in convex analysis, since it allows one to keep lower semi-continuity of functions and to assure the existence of solutions to various extremum problems. For more details on this topic, we refer the reader to \cite{Auslender2003, Dur2017, Liu2018, Pataki2007, Pataki2001}.

It is well-known that the linear image of a closed convex set is not necessarily closed. In fact, this does not necessarily hold even for closed convex cones. On the other hand, it is shown in \cite{Borwein2009} that for a given closed convex cone $X$ in $\mathbb{R}^n,$ the set
\begin{eqnarray*}
\mathrm{int} (\{T \in L(\mathbb{R}^n, \mathbb{R}^m) \ : \ T(X) \text{ is closed}\})
\end{eqnarray*}
is dense and open in $L(\mathbb{R}^n, \mathbb{R}^m)$-the space of all linear mappings from $\mathbb{R}^n$ into $\mathbb{R}^m.$ This result is refined in \cite{Borwein2010}, where it is proved that
$$L(\mathbb{R}^n, \mathbb{R}^m) \setminus \mathrm{int}(\{T\in L(\mathbb{R}^n, \mathbb{R}^m) \ : \ T(X) \text{ is closed}\})$$
is $\sigma$-porous in $L(\mathbb{R}^n, \mathbb{R}^m),$ i.e. small with regard to both measure and category. (See the next section for notation and definitions.)

The aim of this paper is to improve the above two results by weakening the assumption that $X$ is a closed convex cone to the assumption that $X$ is a closed convex set.
More precisely, the main result of the paper is the following.

\begin{theorem}\label{OpenDense}
Let $X \subset \mathbb{R}^n$ be a closed convex set. Then the set
$$L(\mathbb{R}^n, \mathbb{R}^m) \setminus \mathrm{int}(\{T\in L(\mathbb{R}^n, \mathbb{R}^m) \ : \ T(X) \text{ is closed}\})$$
is $\sigma$-porous in $L(\mathbb{R}^n, \mathbb{R}^m).$ In particular, the set
$$\mathrm{int}(\{T\in L(\mathbb{R}^n, \mathbb{R}^m) \ : \ T(X) \text{ is closed}\})$$
is dense and open in $L(\mathbb{R}^n, \mathbb{R}^m).$
\end{theorem}

The proof of Theorem~\ref{OpenDense} will be divided into two steps. Firstly, in terms of asymptotic cones, we provide some sufficient conditions under which
the closedness of the image of a closed convex set under a linear mapping from $\mathbb{R}^n$ into $\mathbb{R}^m$ is preserved under small linear perturbations of the linear mapping. Secondly, we show that
these sufficient conditions hold generically.

The rest of the paper is organized as follows. In Section~\ref{SectionPreliminary}, we present some preliminaries which will be used later.
The definition and some properties of asymptotic cones are given in Section~\ref{Section3}. The proof of Theorem~\ref{OpenDense} will be provided in Section~\ref{Section4}.

\section{Preliminaries} \label{SectionPreliminary}

\subsection{Notation}

Let $\mathbb{R}^n$ denote the Euclidean space of dimension $n.$ The corresponding inner product (resp., norm) in $\mathbb{R}^n$  is defined by $\langle x, y \rangle$ for any $x, y \in \mathbb{R}^n$ (resp., $\| x \| := \sqrt{\langle x, x \rangle}$ for any $x \in \mathbb{R}^n$). The open ball centered at $x \in \mathbb{R}^n$ and of radius $r$ is denoted by ${B}_r^n(x)$, or simply ${B}_r(x)$ if it does not lead to a misunderstanding. As usual, $\mathrm{dist}(x, X)$ denotes the Euclidean distance from $x \in \mathbb{R}^n$ to $X \subset \mathbb{R}^n,$ i.e.,
\begin{eqnarray*}
\mathrm{dist}(x, X) &:=& \inf\{\|x - y \| \ : \ y \in X \}.
\end{eqnarray*}

For an arbitrary set $X \subset \mathbb{R}^n,$ $\overline X$ and $\mathrm{int}(X)$ stand for the closure and the interior of $X$ respectively; the {\em affine hull} of $X,$ denoted by $\mathrm{aff}(X),$ is the intersection of all affine subspaces containing $X;$ the {\em relative interior} of $X$, denoted by $\ri(X)$, is defined by
\begin{eqnarray*}
\ri(X) &:=& \{x \in \mathrm{aff}(X) \ : \ \exists r > 0 \text{ such that } {B}_r(x) \cap \aff(X)\subset X\}.
\end{eqnarray*}
Observe that the relative interior of a convex set $X$ is convex, and is nonempty if $X$ is nonempty. Furthermore, if $X$ is a cone, then so is $\ri(X).$

Let $L(\mathbb{R}^n, \mathbb{R}^m)$ denote the set of all linear mappings from $\mathbb{R}^n$ into $\mathbb{R}^m,$ and we will assume that $L(\mathbb{R}^n, \mathbb{R}^m)$ is equipped with the operator norm.

\subsection{$\sigma$-porous sets}

We present some properties on porosity which will be necessary in the proof of Theorem~\ref{OpenDense}. Let us start with the following definition.

\begin{definition}
Let $X \subset \mathbb{R}^n.$ For any $x \in \mathbb{R}^n$ and any $R > 0,$ set
\begin{eqnarray*}
\gamma(x,R,X) &:=& \sup\{r\geqslant 0\ : \ \exists x' \in \mathbb{R}^n \text{ such that } {B}_r(x') \subset {B}_R(x) \setminus X\}
\end{eqnarray*}
(where we put $\sup \emptyset := 0$). Then {\em the porosity of $X$ at $x$} is defined by
\begin{eqnarray*}
p(x, X) &:=& \liminf_{R\to 0^+}\frac{\gamma(x, R, X)}{R}.
\end{eqnarray*}
The set $X$ is said to be {\em porous} if $p(x, X) > 0$ for every $x \in X.$ Finally, we say that $X$ is {\em $\sigma$-porous} if it is a countable union of
porous sets.
\end{definition}

Clearly, any $\sigma$-porous set in $\mathbb{R}^n$ is a set of the first Baire category and is also of Lebesgue measure zero; we refer the reader to the survey \cite{Zajicek2005} for more details. 

The following statement was formulated without proof in \cite[Proposition 2.1]{Borwein2010}. We provide the proof here for the sake of completeness. 

\begin{lemma}\label{Lemma21}
Let $f \colon \mathbb{R}^n \to \mathbb{R}^m$ be a surjective linear mapping. Then $f^{-1}(Y)$ is porous (resp., $\sigma$-porous) in $\mathbb{R}^n$ whenever $Y$ is porous (resp., $\sigma$-porous) in $\mathbb{R}^m.$ In particular, if $f$ is a linear isomorphism, then $f^{-1}(Y)$ is porous (resp., $\sigma$-porous) if, and only if, $Y$ is porous (resp., $\sigma$-porous). 
\end{lemma}
\begin{proof}
Observe that the second statement follows immediately from the first statement so it is sufficient to prove the first one. Moreover, it is clear that the case $Y$ being $\sigma$-porous follows directly from the case $Y$ being porous, so it remains to consider the case $Y$ is porous.

Since the linear mapping $f$ is surjective, the restriction $f|_{(\ker f)^\perp}$ is a linear isomorphism. Denote by $\nu_f$ the length of the smallest semi-axis of the ellipsoid image of the unit ball in $(\ker f)^\perp$ by $f.$ Then it is not hard to check that $\nu_f = \|(f|_{(\ker f)^\perp})^{-1}\|^{-1} > 0.$ Therefore, for any $x_1, x_2 \in\mathbb R^n,$ we have
\begin{equation}\label{Eqn1}
\nu_f  \|x_1 - x'\| \leqslant {\|f(x_1) - f(x_2)\|},
\end{equation}
where $x'$ is the orthogonal projection of $x_1$ on the affine space $f^{-1}(f(x_2)).$  For simplicity of notation, we let ${c} := \min\{\nu_f, 1\} \leqslant 1$  and $M := \max\big\{\|f\|, \frac{1}{{c}}\big\} \geqslant 1.$

Now let $x \in f^{-1}(Y)$ be an arbitrary point and set $y := f(x).$ By the assumption, $p(y, Y) > 0,$ so for $0 < R \ll 1,$ there is $y' \in \mathbb{R}^m$ such that 
$${B}^m_{\frac{R{c} p(y, Y)}{2}}(y') \subset {B}^m_{R{c}}(y) \setminus Y.$$
Let $x'$ be the orthogonal projection of $x$ on the affine space $f^{-1}(y').$ It follows from \eqref{Eqn1} that
\begin{eqnarray*}
\|x' - x\| & \leqslant & \frac{\|y' - y\|}{{\nu_f}}.
\end{eqnarray*}
Hence for any $u \in {B}^n_{\frac{R p(y, Y)}{2M}}(x'),$ we have  
\begin{eqnarray*}
\|u - x\| \ \leqslant\ \|u - x'\| + \|x' - x\| & < & \frac{R p(y, Y)}{2M} + \frac{\|y' - y\|}{{\nu_f}}\\
& \leqslant& \frac{R p(y, Y)}{2M} + \frac{\|y' - y\|}{{c}} \\
& < & \frac{R p(y, Y)}{2M} + \frac{R{c} - \frac{R{c} p(y, Y)}{2}}{c} \\
& = & R + \frac{R p(y, Y)}{2M} - \frac{R p(y, Y)}{2}\ \leqslant \ R.
\end{eqnarray*}
Thus 
\begin{eqnarray}\label{Eqn2}
{B}^n_{\frac{Rp(y, Y)}{2M}}(x') &\subset& {B}^n_R(x).
\end{eqnarray}
On the other hand, we have
\begin{eqnarray*}
f\Big({B}^n_{\frac{R {c} p(y, Y)}{2M}}(x') \Big) &=& \frac{R {c} p(y, Y)}{2M}f\Big({B}^n_1(x')\Big)\ \subset \ \frac{R {c} p(y, Y)}{2M}\|f\|{B}^m_1(y')\\
&\subset&\frac{R {c} p(y, Y)}{2}{B}^m_1(y') \ = \ {B}^m_{\frac{R {c} p(y, Y)}{2}}(y') \ \subset \ {B}^m_{R {c}} (y)\setminus Y.
\end{eqnarray*}
Consequently, 
\begin{eqnarray*}
{B}^n_{\frac{R {c} p(y, Y)}{2M}}(x') \cap f^{-1}(Y) &=& \emptyset.
\end{eqnarray*}
This, together with \eqref{Eqn2} and the fact that ${c} \leqslant1,$ gives
$${B}^n_{\frac{R {c} p(y, Y)}{2M}}(x') \subset {B}^n_R(x)\setminus f^{-1}(Y),$$
which yields $p(x, f^{-1}(Y)) \geqslant \frac{{c} p(y, Y)}{2M} > 0,$ and hence the lemma. 
\end{proof}

By Lemma~\ref{Lemma21}, the notion of porosity (and $\sigma$-porosity) in finite dimensional normed linear spaces does not depend on the particular choice of norm.
Furthermore, we also have the following useful properties.

\begin{lemma}\label{Lemma22}
If $X \subset \mathbb{R}^n$ is porous (resp., $\sigma$-porous), then any subset of $X$ is porous (resp., $\sigma$-porous).
\end{lemma}
\begin{proof}
This is straightforward.
\end{proof}

\begin{lemma}\label{Lemma23}
The following statements hold:
\begin{enumerate}[{\rm (i)}]
\item Let $X \subset \mathbb{R}^n$ be a $C^1$-manifold of dimension less than $n.$ Then $X$ is porous.

\item For a non-constant polynomial function $P \colon \mathbb{R}^n \rightarrow \mathbb{R},$ the zero set $P^{-1}(0) \subset \mathbb{R}^n$ is porous.
\end{enumerate}
\end{lemma}
\begin{proof}
(i) Let $d := \dim X < n.$ Take any $x \in X.$ By definition, there is an open set $U \subset \mathbb{R}^n$ containing $x,$ an open set $V \subset \mathbb{R}^n,$ and a diffeomorphism $f \colon U \to V$ such that 
\begin{eqnarray*}
f(U \cap X) 
&=& V \cap (\mathbb{R}^d \times \{0\}).
\end{eqnarray*}
Shrinking $U$ if necessary, we may assume that $f$ is bi-Lipschitz on $U,$ i.e., there exist constants $c_1 > 0$ and $c_2 > 0$ such that
\begin{eqnarray*}
c_1 \|x_1 - x_2\|  &\leqslant& \|f(x_1) - f(x_2) \|  \ \leqslant\ c_2 \|x_1 - x_2\|  \quad \textrm{ for all } \quad x_1, x_2 \in U.
\end{eqnarray*}
Let $R > 0$ be such that ${B}_{(\frac{c_1}{2c_2} + 1)R}(x) \subset U$ and ${B}_{c_1R}(f(x)) \subset V.$ Take any $v \in {B}_{c_1R}(f(x)).$ Then $v = f(u)$ for some $u \in U,$ and so
\begin{eqnarray*}
c_1 \|u - x\|  &\leqslant& \|f(u) - f(x) \|  \ = \ \|v - f(x)\| \ < \ c_1R.
\end{eqnarray*}
Hence, $u \in {B}_R(x).$ Since this holds for arbitrary $v = f(u)$ in ${B}_{c_1R}(f(x)),$ we obtain
\begin{eqnarray*}
{B}_{c_1R}(f(x)) &\subset& f ({B}_R(x)).
\end{eqnarray*}
On the other hand, clearly, there is $y' \in {B}_{c_1R}(f(x))$ such that
\begin{eqnarray*}
{B}_{\frac{c_1R}{2}}(y') &\subset& {B}_{c_1R}(f(x)) \setminus (\mathbb{R}^d \times \{0\}).
\end{eqnarray*}
Therefore, 
\begin{eqnarray*}
g({B}_{\frac{c_1R}{2}}(y') ) &\subset& g({B}_{c_1R}(f(x)) ) \setminus X \ \subset \ {B}_R(x) \setminus X,
\end{eqnarray*}
where $g \colon V \to U$ stands for the inverse of $f.$ 

Let $x' := g(y') \in U,$ i.e., $f(x') = y',$ and take any $u \in {B}_{\frac{c_1R}{2c_2}}(x').$ We have
\begin{eqnarray*}
\|u - x \| &\leqslant & \|u - x'\| + \|x' - x\| \ < \ \frac{c_1R}{2c_2} + \frac{\|f(x') - f(x)\|}{c_1} \ < \ \frac{c_1R}{2c_2} + R,
\end{eqnarray*}
and hence $u \in {B}_{(\frac{c_1}{2c_2} + 1)R}(x) \subset U.$ In particular, $v := f(u) \in f(U) = V.$ Observe that
\begin{eqnarray*}
\|v - y' \| & =  & \|f(u) - f(x')\| \ \leqslant \ c_2 \|u - x'\|  \ < \ \frac{c_1R}{2},
\end{eqnarray*}
which yields $v \in {B}_{\frac{c_1R}{2}}(y').$ Consequently, $u = g(v) \in g({B}_{\frac{c_1R}{2}}(y')).$ Since this holds for arbitrary $u$ in ${B}_{\frac{c_1R}{2c_2}}(x'),$ we get
\begin{eqnarray*}
{B}_{\frac{c_1R}{2c_2}}(x') &\subset& g({B}_{\frac{c_1R}{2}}(y') ) \ \subset \ {B}_R(x) \setminus X.
\end{eqnarray*}
By definition, then
\begin{eqnarray*}
\gamma(x, R, X) &\geqslant& \frac{c_1R}{2c_2},
\end{eqnarray*}
and so $p(x, X) \geqslant \frac{c_1}{2c_2} > 0.$ Therefore, $X$ is porous.

(ii) In fact, by the Cell Decomposition Theorem (see, for example, \cite{Benedetti1990, Bochnak1998, HaHV2017}), the set $P^{-1}(0)$ can be represented as a (disjoint) finite union of sets $X_i,$ where each $X_i$ is a $C^1$-manifold in $\mathbb{R}^n$ of dimension $d_i \in \mathbb{N}.$ Since $P$ is non-constant, $d_i < n$ for every $i.$ Hence, all the sets $X_i$ are porous and so is $P^{-1}(0).$
\end{proof}

With a simpler proof, the following result extends \cite[Theorem~2.2]{Borwein2010}.
\begin{lemma}\label{Lemma24}
Let $m, n \in \mathbb{N}$ and let $Y\subset \mathbb{R}^n$ be a linear subspace. Then the set
$$\{T\in L(\mathbb{R}^n, \mathbb{R}^m) \ : \ T|_Y \text{ does not have maximal rank}\}$$
is porous in $L(\mathbb{R}^n, \mathbb{R}^m).$
\end{lemma}
\begin{proof}
Set $p:=\dim Y$. Evidently, there is an orthonormal coordinate system $\{x_1,\dots,x_n\}$ on $\mathbb{R}^n$ such that
\begin{eqnarray*}
Y &=&\{x \in \mathbb{R}^n \ : \ x_{p+1} = \dots = x_n = 0\}.
\end{eqnarray*}

Let $M(m,n)$ be the set of all $m\times n$ matrices (over $\mathbb{R}$). For $A\in M(m,n),$ set $|||A|||:=\max\{\|Ax\|:\ \|x\|=1\};$ then $(M(m,n), ||| \cdot |||)$ is a finite dimensional normed linear space. Consider the linear mapping
$$f \colon M(m,n)\to L(\mathbb{R}^n, \mathbb{R}^m), \quad A \mapsto f(A),$$ 
defined by $[f(A)](x)=A(x).$ Then $f$ is a linear isomorphism. For $T\in L(\mathbb{R}^n, \mathbb{R}^m),$ writing $f^{-1}(T)=:A=(a_{ij})_{\substack{i=1,\ldots,m\\ j=1,\ldots,n}}$, we can see that $T|_Y$  does not have maximal rank if and only if $A|_Y:=(a_{ij})_{\substack{i=1,\ldots,m\\ j=1,\ldots,p}}$ does not have maximal rank. Let $P_1,\dots,P_q$ be the minors of $A|_Y$ by either deleting $p-m$ columns of $A|_Y$ if $p\geqslant m$ or deleting $m-p$ rows of $A|_Y$ if $m>p$. Each $P_i,\ i=1,\dots,q,$ is a polynomial with entries of $A|_Y$ as variables, i.e.,
$$P_i=P_i(a_{11},\dots,a_{1p},a_{21},\dots,a_{pp}).$$
Moreover, we can also see $P_i$ as a polynomial in $mn$ variables
$$P_i=P_i(a_{11},\dots,a_{1p},a_{21},\dots,a_{nn}),$$
of course, some variables may not appear in the expression of $P_i$. It is well-known that $A|_Y$ does not have maximal rank if and only if 
$$P_1 = \cdots = P_q = 0,$$ 
or equivalently,
$$P_1^2 + \cdots + P_q^2 =0.$$ 
Therefore, in view of Lemma~\ref{Lemma21}, to prove the proposition, it is enough to show that the algebraic set defined by $P_1^2 + \cdots + P_q^2 =0$ in $(M(m,n),|||\cdot|||)$ is porous. Since
all norms on a finite dimensional space are equivalent, by Lemma~\ref{Lemma21} again, we can replace the norm $|||\cdot|||$ by the Euclidean norm on $M(m,n)$, which is identified with $\mathbb{R}^{m \times n}.$ Now the desired conclusion follows from Lemma~\ref{Lemma23}.
\end{proof}

\section{Asymptotic cones} \label{Section3}
In this section, we present some sufficient conditions to ensure that the closedness of the image of a closed convex set under a linear mapping from $\mathbb{R}^n$ into $\mathbb{R}^m$ is preserved under small linear perturbations of the linear mapping.  To do this, we need the following concept of asymptotic cone, which seems to have appeared first in the literature in the works of Steinitz~\cite{Steinitz19131916}. For more details on this notion, we refer the reader to the papers \cite{Dedieu1977, Dieudonne1966} and the books \cite{Auslender2003, Rockafellar1970, Rockafellar1998} with references therein.

For a given set $X\subset \mathbb{R}^n$, the {\em asymptotic cone} of $X,$ denoted by $C_\infty X,$ is defined by
\begin{eqnarray*}
C_\infty X &:=& \big\{v \in \mathbb R^n \ : \ \exists x^k\in X, \exists t_k\in(0,+\infty) \textrm{ s. t. } \|x^k\|\to \infty, t_kx^k\to v \text{ as }k\to\infty\big\}.
\end{eqnarray*}
From the definition, we deduce immediately that $C_\infty X$ is a closed cone (not necessarily convex) and that $\mathrm{aff} (C_\infty X)$ is a linear subspace of $\mathbb{R}^n.$ Furthermore, we have the following property.

\begin{lemma}[{see \cite[Proposition~2.1.5]{Auslender2003}}] \label{Lemma31}
Let $X$ be a nonempty convex set in $\mathbb{R}^n.$ Then the asymptotic cone $C_\infty X$ is a closed convex cone.
\end{lemma}

From now on, the term ``ray'' means {\em open ray emanating from the origin $0 \in \mathbb{R}^n$}, i.e., we consider only {\em rays with the endpoint $0$ but $0$ is not included}.

\begin{lemma}\label{Lemma32}
Let $X \subset \mathbb{R}^n$ be a convex set such that $0 \in \ri(X).$ Then $C_\infty X \subset X.$
\end{lemma}

\begin{proof}
Replacing $\mathbb{R}^n$ by $\aff(C_\infty X)$ if necessary, we may suppose that $\dim \aff(C_\infty X) = n.$
Let $\ell$ be a ray in $C_\infty X$. By contradiction, assume that $\ell\not\subset X$. Since $X$ is convex, there is a unique point  $x\in \ell \cap(\overline X\setminus\ri(X))$, i.e., $x$ is the point such that $[0,x)\subset X$ and $(\ell\setminus[0,x]) \cap X = \emptyset.$ Let $H_x$ be a supporting hyperplane for $X$ at $x$ and let $H_0$ be the linear subspace of dimension $n-1$ parallel to $H_x$. Denote by $H_x^*$ and $H_0^*$ the closed half spaces not containing $\ell\setminus[0,x]$ and $\ell$ respectively. Then evidently, $X\subset H_x^*$, so
\begin{eqnarray*}
C_\infty X &\subset & C_\infty H_x^* \ = \ H_0^* \ \not\supset \ \ell.
\end{eqnarray*}
Consequently $\ell\not\subset C_\infty X,$ which is a contradiction.
\end{proof}

\begin{lemma}\label{Lemma33}
Let $X \subset \mathbb{R}^n$ and $x \in \mathbb{R}^n.$ Then $C_\infty X = C_\infty (X - \{x\}),$ where $X - \{x\} :=\{y - x \ : \ y \in X\}.$ Moreover, for any $T\in L(\mathbb{R}^n,\mathbb{R}^m)$, $T(X)$ is closed if and only if $T(X-\{x\})$ is closed.
\end{lemma}

\begin{proof}
The proof is straightforward.
\end{proof}

\begin{remark}
Lemma~\ref{Lemma33} permits us to bring the study of the closedness of linear images of arbitrary convex sets to case of convex sets containing the origin. Indeed, it is enough to pick any point $x\in\ri(X)$ and consider the set $X-\{x\}$ which clearly contains the origin.
\end{remark}

\begin{lemma}[{cf. \cite[Proposition~3]{Borwein2009}}, \cite{Dedieu1978}, {\cite[Corollary~2.3.2]{Auslender2003}}] \label{Lemma34}
Let $X\subset \mathbb{R}^n$ be a closed set (not necessarily convex) and let $T \in L(\mathbb{R}^n, \mathbb{R}^m).$ If
\begin{eqnarray*}
C_\infty X \cap \ker(T) &=& \{0\},
\end{eqnarray*}
then there exists an open neighborhood $\mathcal N$ of $T$ in $L(\mathbb{R}^n, \mathbb{R}^m)$ such that for any $S\in\mathcal N,$ we have
\begin{eqnarray*}
C_\infty X \cap \ker(S) &=& \{0\}
\end{eqnarray*}
and $S(X)$ is closed in $\mathbb{R}^m.$
\end{lemma}

\begin{proof}
Let $C := \{v\in C_\infty X \ : \ \|v\|=1\}$. Since $C_\infty X$ is closed, both $C$ and $T(C)$ are compact. By the assumption, $0 \not \in T(C).$ Therefore, $\mathrm{dist}(0, T(C)) > 0.$ Consequently, there exists an open neighborhood $\mathcal N$ of $T$ in $L(\mathbb{R}^n, \mathbb{R}^m)$ such that $\mathrm{dist}(0,S(C)) > 0$ for any $S\in\mathcal N,$ i.e., $C_\infty X \cap \ker(S) = \{0\}$. Now fix
a linear mapping $S\in\mathcal N$ and let $x^k \in X$ be a sequence such that $S(x^k)\to y;$ we need to show that $y\in S(X)$. First of all, assume that the sequence $x^k$ is unbounded. Taking a subsequence if necessary, we may suppose that $\|x^k\|\to\infty$ and the sequence $\frac{x^k}{\|x^k\|}$ is convergent to a limit $v.$ Then $v\in C$ by definition. Moreover, we have
\begin{eqnarray*}
\|S(v)\| &=& \left \|\lim_{k\to+\infty}S \Big(\frac{x^k}{\|x^k\|}\Big) \right\| \ = \ \left \|\lim_{k\to+\infty}\frac{S(x^k)}{\|x^k\|} \right \| \ = \ \frac{\|y\|}{\lim_{k\to+\infty}\|x^k\|}\ = \ 0.
\end{eqnarray*}
Hence $v\in C \cap \ker(S),$ which contradicts the fact that $\mathrm{dist}(0, S(C)) > 0.$ Therefore, the sequence $x^k$ is bounded, and so it has a cluster point, say $x.$ Clearly, $x \in X$ (since $X$ is closed) and $y = S(x) \in S(X).$ The lemma is proved.
\end{proof}

\begin{lemma}\label{Lemma35}
Let $X \subset \mathbb{R}^n$ be a closed convex set and let $T \in L(\mathbb{R}^n, \mathbb{R}^m)$ be such that the restriction of $T$ on the linear subspace $Y := \mathrm{aff}(C_\infty X)$ has rank $m.$ If
\begin{eqnarray*}
\ri(C_\infty X)\cap\ker(T)  &\ne& \emptyset,
\end{eqnarray*}
then there exists an open neighborhood $\mathcal{N}$ of $T$ in $L(\mathbb{R}^n, \mathbb{R}^m)$ such that
\begin{eqnarray*}
\ri(C_\infty X)\cap\ker(S) &\ne & \emptyset
\end{eqnarray*}
for any $S \in \mathcal{N}.$ In addition, $S(X)$ is closed in $\mathbb{R}^m$ for each $S\in\mathcal{N}.$
\end{lemma}

\begin{proof}
According to \cite[Lemma 1]{Borwein2009}, it suffices to show that $T(X)$ is closed. In view of Lemma~\ref{Lemma33}, we may assume that $0\in \ri(X)$. Then $C_\infty X\subset X$ by Lemma~\ref{Lemma32}. So in order to prove that $T(X)$ is closed, it is enough to show that $T(C_\infty X)=\mathbb{R}^m$. By the assumption, there exists a ray $\ell\subset\ri(C_\infty X) \cap \ker(T)$. Moreover, there is a real number $\delta \in (0, 1)$ such that
 \begin{eqnarray*}
N_\delta(\ell) &:=& \{x \in Y\setminus\{0\}\ : \ \widehat{\ell,\ell_x}\leqslant\frac{\pi}{2} \quad \textrm{ and } \quad \sin\widehat{\ell,\ell_x}\leqslant\delta\} \ \subset \ C_\infty X,
\end{eqnarray*}
where $\ell_x$ denotes the open ray emanating from the origin to $x,$ namely, $\ell_x=\{rx \ : \ r > 0\}$ and $\widehat{\ell, \ell_x}$ denotes the angle between $\ell$ and $\ell_x.$
Take any $y\in \mathbb{R}^m.$ The assumption that $\rank(T|_Y) = m$ implies that the affine space $(T|_Y)^{-1}(y) = T^{-1}(y) \cap Y$ is non empty. Let $x$ be the orthogonal projection of the origin $0 \in \mathbb{R}^n$ on $(T|_Y)^{-1}(y)$, $v$ be the unit direction of $\ell$ and $\displaystyle t>\|x\|\sqrt{\frac{1-\delta^2}{\delta^2}}.$ By definition, we have $x + tv\in (T|_Y)^{-1}(y)$ and $x\perp v,$ so
\begin{eqnarray*}
\sin\widehat{\ell,\ell_{x+tv}} &=& \frac{\|x\|}{\|x+tv\|} \ = \ \frac{\|x\|}{\sqrt{\|x\|^2 + t^2}} \ < \ \frac{\|x\|}{\sqrt{\|x\|^2+\displaystyle\frac{1-\delta^2}{\delta^2}\|x\|^2}} \ =\ \delta.
\end{eqnarray*}
Hence $x+tv\subset N_\delta(\ell)\subset C_\infty X,$ which yields $y = T(x + tv) \in T(C_\infty X).$ Thus $T(C_\infty X) = \mathbb{R}^m$ and the lemma follows.
\end{proof}

\begin{corollary}\label{Corollary31}
Let $X \subset \mathbb{R}^n$ be a closed convex set and let $Y := \mathrm{aff}(C_\infty X).$ Designate by $\mathcal M'\subset L(\mathbb{R}^n, \mathbb{R}^m)$ the family of all linear mappings $T$ such that the restriction $T|_Y$ has rank $m$. Then the set
\begin{eqnarray*}
\mathcal{H} &:=& \big \{T\in \mathcal M':\ C_\infty X\cap\ker (T)=\{0\}\ \text { or }\ \ri(C_\infty X)\cap\ker(T)\not=\emptyset \big \}
\end{eqnarray*}
is an open dense set in $L(\mathbb{R}^n, \mathbb{R}^m)$. Moreover $\mathcal{H} \subseteq  \mathrm{int}(\{T\in \mathcal M' \ : \ T(X) \text{ is closed}\}).$
\end{corollary}

\begin{proof}
Note that the set $\mathcal M'$ is dense and open in $L(\mathbb{R}^n, \mathbb{R}^m)$. Hence in light of Lemmas~\ref{Lemma34}~and~\ref{Lemma35}, it is clear that $\mathcal{H}$ is open and the last statement holds. So it remains to show that $\mathcal{H}$ is dense. Clearly, it is enough to prove that $\mathcal{H}$ is dense in $\mathcal M'$. To see this, take any $T\in \mathcal M'\setminus \mathcal{H}.$ Evidently, we have
\begin{eqnarray*}
\{0\} & \ne &  C_\infty X \cap \ker(T) \ \subset \ C_\infty X\setminus\ri(C_\infty X).
\end{eqnarray*}
Pick $v^* \in (C_\infty X\cap\ker(T))\setminus(\{0\}\cup\ri(C_\infty X))$ such that $\|v^*\|=1$. By definition, there is a sequence $v^k \in \ri(C_\infty X)$ with $\lim_{k \to \infty} v^k = v^*$, so
\begin{eqnarray*}
\lim_{k \to \infty} T(v^k) &=& T(v^*) \ = \ 0.
\end{eqnarray*}
With no loss of generality, we can suppose that $\|v^k\|=1$. 
For each $k$, let $\pi_k^1 \colon \mathbb R^n\to (v^*)^\perp$ and $\pi_k^2 \colon \mathbb R^n\to (v^k)^\perp$ designate the orthogonal projections. Note that for $k$ large enough, $\pi_k^2|_{(v^*)^\perp}$ is a linear isomorphism, so set $T_k:=T\circ (\pi_k^2|_{(v^*)^\perp})^{-1}\circ\pi_k^2$. Clearly $\pi_k^2(v^k)=0$, thus $T_k(v^k)=0.$ Now for any $x\in\mathbb R^n$ with $\|x\|=1$, we have 
$$\begin{array}{lll}
\|T_k(x)-T(x)\|&=&\|(T\circ (\pi_k^2|_{(v^*)^\perp})^{-1}\circ\pi_k^2)(x)-T(x)\| \\
&=&\|(T\circ (\pi_k^2|_{(v^*)^\perp})^{-1}\circ\pi_k^2)(x)-T(\pi_k^1(x))\| \\
&\leqslant& \|T\| \cdot \|(\pi_k^2|_{(v^*)^\perp})^{-1}(\pi_k^2(x))-\pi_k^1(x)\|\leqslant\|T\|\tan \widehat{v^k,v^*}.
\end{array} $$
Consequently, for $k$ large enough, the linear mapping $T_k$ 
belongs to $\mathcal{H}$ and we have $\|T_k-T\|\to 0$ as $k\to+\infty.$ This ends the proof of the corollary.
\end{proof}

\section{Proof of the main result}\label{Section4}

We can now complete the proof of Theorem~\ref{OpenDense}.

\begin{proof}[Proof of Theorem~\ref{OpenDense}]
Set $Y := \aff(C_\infty X)$, the smallest linear subspace containing $C_\infty X$ and let $p:=\dim Y.$ Since the closedness of a linear mapping from $\mathbb{R}^n$ into $\mathbb{R}^m$ is invariant by making coordinate changes, we can suppose that
\begin{eqnarray*}
Y &:=& \{x \in \mathbb{R}^n \ : \ x_{p+1}=\dots=x_n=0\} \ = \ \mathbb{R}^p\times\{(0,\dots,0)\}
\end{eqnarray*}
and identify $Y$ with $\mathbb{R}^p$. Denote by $\mathcal M\subset L(\mathbb{R}^n, \mathbb{R}^m)$ and $\mathcal M'\subset L(\mathbb{R}^n, \mathbb{R}^m)$, respectively, the family of all linear mappings $T$ of maximal rank (i.e., $\rank (T) = \min\{m,n\}$) and the family of all linear mappings $T$ such that the restriction $T|_Y$ has maximal rank (i.e., $\rank (T|_Y) = \min\{m,p\}$). Let $\mathcal M''\subset L(Y,\mathbb{R}^m)=L(\mathbb{R}^p, \mathbb{R}^m)$ designate the family of all linear mappings $T\colon Y\to \mathbb{R}^m$ of maximal rank.
It is easy to verify that $\mathcal M$ and $\mathcal M'$ are open dense subsets of $L(\mathbb{R}^n, \mathbb{R}^m),$ and $\mathcal M''$ is an open dense subset of $L(Y,\mathbb{R}^m).$ We consider three cases:

\subsubsection*{Case 1: $\dim Y\leqslant m$}
 In view of Lemma~\ref{Lemma24}, the set $L(\mathbb{R}^n, \mathbb{R}^m)\setminus \mathcal M'$ is $\sigma$-porous in $L(\mathbb{R}^n, \mathbb{R}^m).$ So to prove that the set
$$L(\mathbb{R}^n, \mathbb{R}^m)\setminus \mathrm{int}( \{T\in L(\mathbb{R}^n, \mathbb{R}^m) \ : \ T(X) \text{ is closed}\})$$
is $\sigma$-porous in $L(\mathbb{R}^n, \mathbb{R}^m)$, it is enough to show that the set
$$\mathcal M'\setminus \mathrm{int}(\{T\in L(\mathbb{R}^n, \mathbb{R}^m) \ : \ T(X) \text{ is closed}\})$$
is $\sigma$-porous in $L(\mathbb{R}^n, \mathbb{R}^m)$. Since $\dim Y\leqslant m$, for each $T\in\mathcal M'$, the restriction $T|_Y$ is one-to-one and so $C_\infty X \cap \ker (T|_Y) = \{0\}.$  Note that $C_\infty X\subset Y$, so it follows easily that $C_\infty X \cap \ker (T) = \{0\}$ for each $T\in \mathcal M',$ and thus we are done by Lemma~\ref{Lemma34}.

\subsubsection*{Case 2: $\dim Y>m$ and $Y=\mathbb{R}^n$}
In light of Lemma~\ref{Lemma24}, the set $L(\mathbb{R}^n, \mathbb{R}^m)\setminus \mathcal M$ is $\sigma$-porous in $L(\mathbb{R}^n, \mathbb{R}^m)$. Hence to prove the theorem in this case, it is enough to show that the set
$$\mathcal M\setminus \mathrm{int}(\{T\in L(\mathbb{R}^n, \mathbb{R}^m) \ : \ T(X) \text{ is closed}\})$$
is $\sigma$-porous in $L(\mathbb{R}^n, \mathbb{R}^m)$. To this end, set
\begin{eqnarray*}
\mathcal{A} &:=& \{T\in \mathcal M'':\ \{0\}\not=\ri(C_\infty X)\cap\ker(T)\subset C_\infty X\setminus\ri(C_\infty X)\}.
\end{eqnarray*}
Note that $\mathcal M=\mathcal M''$ as $Y = \mathbb{R}^n,$ so
\begin{eqnarray*}
\mathcal{A} &=& \{T\in \mathcal M:\ \{0\}\not=\ri(C_\infty X)\cap\ker(T)\subset C_\infty X\setminus\ri(C_\infty X)\}.
\end{eqnarray*}
From the proof of~\cite[Theorem 3.2]{Borwein2010}, we can see that the set $\mathcal{A}$ is $\sigma$-porous in $L(\mathbb{R}^n, \mathbb{R}^m),$ i.e.,
$\mathcal{A} = \bigcup_{k = 1}^\infty \mathcal{A}_k,$ where each $\mathcal{A}_k$ is a porous set in $L(\mathbb{R}^n, \mathbb{R}^m).$ Let
\begin{eqnarray*}
\mathcal{B} &:=& \{T\in \mathcal M \ : \ T(X) \textrm{ is not closed}\}.
\end{eqnarray*}
By Lemmas~\ref{Lemma34}~and~\ref{Lemma35}, it is clear that $\mathcal M \setminus \mathcal{A} \subseteq \mathcal M\setminus \mathcal{B}.$ Hence $\mathcal{B} \subseteq \mathcal{A}$ and we have $\mathcal{B} = \bigcup_{k = 1}^\infty (\mathcal{A}_k \cap \mathcal{B}).$ Now the proof of the theorem in this case follows from Lemma~\ref{Lemma22}.

\subsubsection*{Case 3: $\dim Y>m$ and $Y\not=\mathbb{R}^n$}
Again, by Lemma~\ref{Lemma24}, it is sufficient to prove that the set
$$\mathcal M'\setminus \mathrm{int}(\{T\in L(\mathbb{R}^n, \mathbb{R}^m) \ : \ T(X) \text{ is closed}\})$$
is $\sigma$-porous in $L(\mathbb{R}^n, \mathbb{R}^m)$. From Case 2, we know that $\mathcal{A}$ is $\sigma$-porous in $L(Y,\mathbb{R}^m).$ Let 
$$\pi \colon L(\mathbb{R}^n, \mathbb{R}^m)\to L(Y,\mathbb{R}^m)=L(\mathbb{R}^p, \mathbb{R}^m)$$
be the linear mapping defined by $\pi(T) := T|_Y.$ Clearly, $\pi$ is surjective.
By Lemma~\ref{Lemma21}, $\pi^{-1}(\mathcal{A})$ is $\sigma$-porous in $L(\mathbb{R}^n, \mathbb{R}^m)$. Note that $\pi^{-1}(\mathcal M'')=\mathcal M'$ and so
\begin{eqnarray*}
\pi^{-1}(\mathcal{A})
&=& \pi^{-1} \big(\{S\in \mathcal M'' \ : \ \{0\}\not=\ri(C_\infty X)\cap\ker(S)\subset C_\infty X\setminus\ri(C_\infty X)\} \big)\\
&=&  \{T\in \pi^{-1}(\mathcal M'') \ : \ \{0\}\not=\ri(C_\infty X)\cap\ker(T)\subset C_\infty X\setminus\ri(C_\infty X)\}\\
&=& \{T\in \mathcal M' \ : \ \{0\}\not=\ri(C_\infty X)\cap\ker(T)\subset C_\infty X\setminus\ri(C_\infty X)\}\\
&=& \mathcal M' \setminus \mathcal{H} \\
&\supseteq& \mathcal{M}' \setminus \mathrm{int}( \{T\in \mathcal M' \ : \ T(X) \text{ is closed}\}),
\end{eqnarray*}
where the set $\mathcal{H}$ and the last inclusion are taken from Corollary~\ref{Corollary31}. Applying Lemma~\ref{Lemma22} again, it follows that that $\mathcal{M}'\setminus \mathrm{int}(\{T\in \mathcal M' \ : \ T(X) \text{ is closed}\})$ is $\sigma$-porous in $L(\mathbb{R}^n,\mathbb{R}^m)$. This ends the proof of the theorem.
\end{proof}

\subsection*{Acknowledgments}
The authors wish to thank the referee(s) for her/his/their careful reading and constructive comments on the manuscript.


\end{document}